\documentclass[a4paper,11pt,english]{article}
\usepackage[utf8]{inputenc}
\usepackage{lineno}
\usepackage[english]{babel}
\usepackage{hyperref}
\usepackage{tikz}
\usepackage{hyphenat}
\usetikzlibrary{calc,arrows}
\usetikzlibrary{patterns}
\usepackage{amsmath,amssymb,amsthm}
\usepackage{enumerate}
\usepackage{url}
\newtheorem{thm}{Theorem}[section]
\newtheorem{dfn}[thm]{Definition}

\newtheorem{prop}[thm]{Proposition}
\newtheorem{cor}[thm]{Corollary}

\newtheorem{quest}[thm]{Question}

\theoremstyle{remark}
\newtheorem{rem}[thm]{Remark}
\newtheorem{ex}[thm]{Example}

\DeclareMathOperator{\scal}{Scal}
\DeclareMathOperator{\ric}{Ric}
\DeclareMathOperator{\R}{R}
\DeclareMathOperator{\Lop}{L}
\DeclareMathOperator{\Wop}{W}

\DeclareMathOperator{\tr}{trace}
\DeclareMathOperator{\Id}{I}
\DeclareMathOperator{\id}{id}
\DeclareMathOperator{\ad}{ad}
\DeclareMathOperator{\Ad}{Ad}

\DeclareMathOperator{\rg}{rank}
\DeclareMathOperator{\ke}{Ker}

\newcommand{\ddt}[1]{\frac{\partial #1 }{\partial t}}
\newcommand{\ps}[2]{\left\langle #1,#2 \right\rangle}
\newcommand{\eps}{\varepsilon}

\newcommand{\textoverline}[1]{$\overline{\mbox{#1}}$}

\title{Curvature cones and the Ricci flow.}
\author{Thomas Richard}
\begin{document}
\maketitle

\begin{abstract}
  This survey reviews some facts about nonnegativity conditions on the
  curvature tensor of a Riemannian manifold which are preserved by the
  action of the Ricci flow. The text focuses on two main points.

  First we
  describe the known examples of preserved curvature conditions and
  how they have been used to derive geometric results, in particular sphere
  theorems.

  We then describe some recent results which give restrictions
  on general preserved conditions.

  The paper ends with some open questions on these matters.
\end{abstract}

The Ricci flow is the following evolution equation:
\begin{equation}
\begin{cases}
  \frac{\partial g}{\partial t}=-2\ric_{g(t)}\\
  g(t)=g_0
\end{cases}\label{eq:RF}
\end{equation}
where $\left(g(t)\right)_{t\in[0,T)}$ is a one-parameter family of smooth
Riemannian metrics on a fixed manifold $M$, and $g_0$ is a given
smooth Riemannian metric on $M$. It was introduced by R. Hamilton in 1982
(\cite{zbMATH03794902}), where it was used to study the topology of compact $3$-manifolds with
positive Ricci curvature. Analytically, the Ricci flow is a degenerate
parabolic system. Existence and uniqueness for the Cauchy problem
\eqref{eq:RF} have been established by
Hamilton in the case where $M$ is compact and $g_0$ is smooth.

For a general introduction to the Ricci flow, see the books \cite{zbMATH02121403},
\cite{zbMATH05062652}, \cite{zbMATH05081815}. Since Hamilton's work,
the Ricci flow has been used to solve various geometric problems. We
refer to the previously cited books for examples. Here we will just briefly mention two of the biggest geometric
achievements of Ricci flow: 
\begin{itemize}
\item The proof of Thurston's Geometrization conjecture for
  $3$-manifolds by
  G. Perelman, (\cite{2002math.....11159P},\cite{2003math......7245P},\cite{2003math......3109P}), for an exposition of Perelman's proof see
  (\cite{zbMATH05188193}, \cite{zbMATH05796326},\cite{zbMATH05530173}).
\item The proof of the differentiable sphere theorem, by Brendle and
  Schoen (\cite{zbMATH05859406}), see also the books (\cite{zbMATH05673930},\cite{zbMATH05821328}).
\end{itemize}
The objects we will be dealing with in this survey have little to do with
Perelman's work, but were pivotal in the proof of the differentiable
sphere theorem.

A priori estimates are among the most basic tools in the study of PDEs,
one can divide them into two classes: integral estimates and pointwise
estimates. Pointwise estimates often come from suitable maximum principles
and are thus more often encountered in the realm of parabolic or elliptic
equations. The Ricci flow being a geometric parabolic PDE, it is
tempting to look for geometrically meaningful pointwise estimates for Ricci flows.

If one looks at Hamilton's foundational work, one sees that after
having proven short time existence and derived some variation formulas
for the Ricci flow, Hamilton proves the following result:
\begin{prop}
  Let $(M^3,g_0)$ be a compact $3$-manifold with nonnegative Ricci
  curvature, then the solution $g(t)$ to \eqref{eq:RF} with initial
  condition $g_0$ satisfies $\ric_{g(t)}\geq 0$ for all $t\geq 0$.
\end{prop}
This is the kind of geometric pointwise estimate we will be
concerned with. More precisely, we will try to gather what is known
about various nonnegativity properties of the curvature of Riemannian
manifold $(M,g_0)$ which remain valid for solutions $g(t)$ of the
Ricci flow starting at $g_0$.

Let us now describe the contents of this paper. Section 1 sets the
scene by introducing an abstract framework which allows us to consider nonnegativity
conditions on the curvature as convex cones inside some vector space satisfying some
invariance properties: the so-called curvature cones. We then describe how Hamilton's maximum
principle characterizes which curvature cones lead to a nonnegativity
condition on the curvature which is preserved under the action of the
Ricci flow. In section 2 we review the various known examples of Ricci
flow invariant curvature conditions, and see how they have been used
to prove various sphere theorems. Section 3 deals with known
restrictions on those Ricci flow invariant curvature conditions.
In section 4 we gather some open questions on these matters.

Let us end this introduction with a disclaimer: this survey doesn't
claim any originality in the treatment of the subject and draws
heavily on the available litterature. Its main purpose is to gather
results which were only available in separate articles before.

\subsection*{Acknowledgements}
\label{sec:acknowledgements}

The author thanks Harish Seshadri for his advised comments on a
preliminary version of this text.

\section{Curvature cones}
\label{sec:curvature-cones}

\subsection{Definition and first properties}
\label{sec:defin-first-prop}
Let us consider a Riemannian manifold $(M,g)$. We recall that its
Riemann curvature tensor $R$ is the $(3,1)$ tensor defined by
\[R(X,Y)Z=\left(\nabla_Y\nabla_X-\nabla_X\nabla_Y-\nabla_{[Y,X]}\right)Z\]
where $\nabla$ is the Levi-Civitta connection canonically built from
the metric $g$ and $X,Y,Z$ are vector fields on $M$. (Beware here that
there is no universal convention for the sign of $R$.)

$R$ enjoys the following symmetries:
\begin{itemize}
\item $R(x,y)z=-R(y,x)z$,
\item $R(x,y)z+R(y,z)x+R(z,x)y=0$,
\item $g(R(x,y)z,t)=-g(R(x,y)t,z)$.
\end{itemize}

Using the symmetries of $R$, on can build at each $p\in M$ a symmetric endomorphism $\R$ of
$\Lambda^2T_pM$ by
\[g(\R(x\wedge y),z\wedge t)=g(R(x,y)z,t) \]
for $x,y,z,t\in T_pM$. $\R$ is called the \emph{curvature operator} of $(M,g)$.
\begin{rem}
  Here, as in the rest of the paper, the inner product $g$ on $\Lambda^2T_pM$ is
  the one which comes from the metric $g$ by the
  following formula:
  \[g(x\wedge y,z\wedge t)=g(x,z)g(y,t)-g(x,t)g(y,z)\]
  (extended to non simple elements of $\Lambda^2T_pM$ by bilinearity).

  We will use the same construction to endow $\Lambda^2\mathbb{R}^n$
  with an inner product $\ps{\ }{\ }$ coming from the standard inner
  product $\ps{\ }{\ }$ on $\mathbb{R}^n$. 
\end{rem}
\begin{dfn}
  The space of algebraic curvature operators
  $S^2_B\Lambda^2\mathbb{R}^n$ is the space of symmetric endomorphisms
  $\R$ of $\Lambda^2\mathbb{R}^n$ which satisfy the first Bianchi identity:
  \[\forall x,y,z,t\in\mathbb{R}^n\quad \ps{\R(x\wedge y)}{z\wedge t}
  +\ps{\R(z\wedge x)}{y\wedge t}+\ps{\R(y\wedge z)}{x\wedge t}=0.\]
\end{dfn}
Remark that $S^2_B\Lambda^2\mathbb{R}^n$ has a natural inner product
given by: \[\ps{\R}{\Lop}=\tr (\R\Lop^T).\]

The space of algebraic curvature operators is the space of 
operators which satisfy the same symmetries as the curvature operators
of Riemannian manifolds. As in the case of Riemannian manifolds, 
it is interesting to consider the Ricci morphism:
$\rho:S^2_B\Lambda^2\mathbb{R}^n\to S^2\mathbb{R}^n$ which associates
to an algebraic curvature operator $\R$ its Ricci tensor which is a
symmetric operator on $\mathbb{R}^n$ defined by:
\[\ps{\rho(\R)x}{y}=\sum_{i=1}^n \ps{\R(x\wedge e_i)}{y\wedge e_i}\]
where $(e_i)_{1\leq i\leq n}$ is an orthonormal basis of
$\mathbb{R}^n$. $\R$ is said to be Einstein if $\rho(\R)$ is a
multiple of the identity operator
$\id:\mathbb{R}^n\to\mathbb{R}^n$. Similarly, the scalar curvature of
an algebraic 
curvature operator is just twice its trace.

The action of $O(n,\mathbb{R})$ on $\mathbb{R}^n$ induces the
following action of $O(n,\mathbb{R})$ on $S^2_B\Lambda^2\mathbb{R}^n$: 
\begin{equation}
\ps{g.\R(x\wedge y)}{z\wedge t}=\ps{\R(gx\wedge gy)}{gz\wedge gt}.\label{eq:action}
\end{equation}

The representation of $O(n,\mathbb{R})$ given by its
action on $S^2_B\Lambda^2\mathbb{R}^n$ is decomposed into irreducible
representations in the following way:
\begin{equation}
S^2_B\Lambda^2\mathbb{R}^n=\mathbb{R}\Id\oplus
(S^2_0\mathbb{R}^n\wedge\id)\oplus \mathcal{W}\label{eq:decomp}
\end{equation}
where the space of Weyl curvature operators $\mathcal{W}$ is the
kernel of the Ricci endomorphism $\rho:S^2_B\Lambda^2\mathbb{R}^n\to
S^2\mathbb{R}^n$ and $S^2_0\mathbb{R}^n\wedge\id$ is the image of the space of traceless
endomorphims of $\mathbb{R}^n$ under the application $A_0\mapsto
A_0\wedge\id$. The wedge product of two symmetric operators
$A,B:\mathbb{R}^n\to \mathbb{R}^n$ is defined by
\[(A\wedge B)(x\wedge y)=\frac{1}{2}\left ( Ax\wedge By+Bx\wedge
  Ay\right).\]

This corresponds to half the Kulkarni-Nomizu product of $A$ and
$B$ viewed as quadratic forms.
In dimension 2, only the first summand of
\eqref{eq:decomp} exists. In dimension $3$ the $\mathcal{W}$
factor is $0$. Starting in dimension 4, all three components exist.

When needed, we will write $\R=\R_{\Id}+\R_0+\R_\mathcal{W}$ the
decomposition of a curvature operator along the three irreducible
components of \eqref{eq:decomp}.

\begin{dfn}
  A curvature cone is a \emph{closed convex} cone $\mathcal{C}\subset
  S^2_B\Lambda^2\mathbb{R}^n$ which is invariant under the action of
    $O(n,\mathbb{R})$ given by \eqref{eq:action}.
\end{dfn}

\begin{dfn}
  A curvature cone is said to be nonnegative if it contains the identity operator
    $\Id:\Lambda^2\mathbb{R}^n\to\Lambda^2\mathbb{R}^n$ in its
    \emph{interior}.
\end{dfn}


This definition can be tracked back to the article \cite{zbMATH00709417} of
M. Gromov. One should notice that we require the cone to be invariant
under the full orthogonal group $O(n,\mathbb{R})$, rather than under
the special orthogonal group $SO(n,\mathbb{R})$. This makes a
difference only in dimension 4, where the action of $SO(4,\mathbb{R})$
on the space of Weyl tensors is not irreducible. The behavior of these
``oriented" curvature cones will be briefly addressed in section \ref{sec:case-dimension-4}.

Each nonnegative curvature cone $\mathcal{C}$ can be used to define a nonnegativity
condition on the curvature 
of Riemannian manifolds. The curvature operator
$\R$ of a Riemannian manifold $(M,g)$ is a section of the bundle $S^2_B\Lambda^2TM$
which is built from $TM$  the same way $S^2_B\Lambda^2\mathbb{R}^n$ is
built from $\mathbb{R}^n$. For each $x\in M$, one can choose an
orthonormal basis of $T_xM$ to build an isomorphism between
$S^2_B\Lambda^2T_xM$ and $S^2_B\Lambda^2\mathbb{R}^n$. Thanks to the
$O(n,\mathbb{R})$-invariance of $\mathcal{C}$, this allows us to embed
$\mathcal{C}$ in $S^2_B\Lambda^2T_xM$ in a way which is independent
of the basis of $T_xM$ we started with.

\begin{dfn}
  Let $\mathcal{C}$ be a nonnegative curvature cone.

  A Riemannian manifold $(M,g)$ has \emph{$\mathcal{C}$-nonnegative
  curvature} if for any $x\in M$ the curvature operator of $(M,g)$ at
  $x$ belongs to the previously discussed embedding of $\mathcal{C}$
  in $S^2_B\Lambda^2T_xM$.

  Similarly, $(M,g)$ has \emph{positive $\mathcal{C}$-curvature} if its
  curvature operator at each point is in the interior of
  $\mathcal{C}$.
\end{dfn}

Let us give a couple of examples of curvature cones:
\begin{itemize}
\item $\mathbb{R}\Id$, $S^2_0\mathbb{R}^n\wedge\id$ and $\mathcal{W}$
  are curvature cones, which are not nonnegative.
\item $\{\R|\tr \R\geq 0\}$ is a nonnegative curvature curvature
  cone. The corresponding nonnegativity condition is ``nonnegative
  scalar curvature''.
\item Similarly the conditions ``nonnegative Ricci curvature'',
  ``nonnegative sectionnal curvature'' and ``nonnegative
  curvature operator'' define nonnegative curvature cones.
\end{itemize}


With these examples in mind we can go back to the definition and try
to explain its significance.
\begin{itemize}
\item The fact that $\mathcal{C}$ is $O(n,\mathbb{R})$-invariant is mandatory to
  be able make sense of the ``$\mathcal{C}$ nonnegative curvature'' condition.
\item Requiring that $\mathcal{C}$ is a cone ensures that the
  associated geometric condition is invariant under scalings, which is
  expected from a nonnegativity condition on the curvature.
\item Asking for $\Id$ to be in the interior of $\mathcal{C}$ is
  equivalent to requiring that the round sphere has positive
  $\mathcal{C}$-curvature. It also ensure that the ``positive
  $\mathcal{C}$-curvature'' condition is stable with respect to $C^2$ perturbations
  on the space of Riemannian metrics.
\item It is not so clear why one should ask for convexity of
  $\mathcal{C}$, however it is satisfied by all classical curvature
 conditions and turns out to be a crucial hypothesis when dealing
  with Hamilton's maximum principle (see next section).  
\end{itemize}

We have the following elementary observation, 
whose elementary proof is not written anywhere as far as the author knows:
\begin{prop}
  If $\mathcal{C}$ is a nonnegative curvature cone which is not the
  full space $S^2_B\Lambda^2\mathbb{R}^n$, then $\mathcal{C}\subset\{\R|\tr \R\geq 0\}$.
\end{prop}
\begin{proof}
Let us assume that $\mathcal{C}$ contains a curvature operator $\R$
with negative trace. Consider the average
$\tilde{\R}=\int_{O(n)}g\cdot\R dg$ of the $O(n)$-orbit of $\R$ with
respect to the Haar measure $dg$ on $O(n)$. 

The irreducibility of
$\mathcal{W}$ and $S^2_0\mathbb{R}^n\wedge\id$ imply that the
projection of $\tilde{\R}$ on these subspaces vanishes, thus
$\tilde{\R}=\lambda\Id$ where $\lambda<0$ since
$\tr\tilde{\R}=\tr\R<0$. Hence $\mathcal{C}$ contains the whole line
$\mathbb{R}\Id$ and since $\Id$ is in the interior of $\mathcal{C}$,
this implies that $\mathcal{C}=S^2_B\Lambda^2\mathbb{R}^n$.
\end{proof}
\subsection{Cones which behave well under the Ricci flow}
\label{sec:cones-which-behave}

We now consider the interplay between these curvature cones and the
Ricci flow. If $(M,g(t))$ is a Ricci flow, Hamilton has proved in \cite{zbMATH04022085} that
the curvature operator $\R_{g(t)}$ of $(M,g(t))$ satisfies the
following evolution equation:
\[\ddt{\R_{g(t)}}=\Delta_{g(t)}\R_{g(t)}+2Q(\R_{g(t)})\]
where $Q$ is the $O(n,\mathbb{R})$ quadratic vector field on
$S^2_B\Lambda^2\mathbb{R}^n$ defined by:
\[Q(\R)=\R^2+\R^\#.\]
Here, $\R^2$ is just the square of $\R$ seen as an endomorphism of
$\Lambda^2\mathbb{R}^n$.  $\R^\#$ is defined in the following way:
\[\ps{\R^\#\eta}{\eta}=-\frac{1}{2}\tr(\ad_\omega\circ\R\circ\ad_\omega\circ\R)\]
where $\ad_\omega:\Lambda^2\mathbb{R}^n\to\Lambda^2\mathbb{R}^n$ is
the endomorphism $\eta\mapsto [\omega,\eta]$. In the previous formula,
the Lie bracket $[\ ,\ ]$ on $\Lambda^2\mathbb{R}^n$ comes from its
identification with $\mathfrak{so}(n,\mathbb{R})$ given by:
\[x\wedge y\mapsto (u\mapsto \ps{x}{u}y-\ps{y}{u}x).\]
This expression for $\R^\#$ can be found in \cite{zbMATH05578712}.

We will sometimes use the bilinear map $B$ associated to the quadratic map
$Q$, it is defined in the usual way:
\[B(\R_1,\R_2)=\frac{1}{2}\bigl(Q(\R_1+\R_2)-Q(\R_1)-Q(\R_2)\bigr). \]

We are now ready to define a Ricci flow invariant curvature cone:

\begin{dfn}
  A curvature cone $\mathcal{C}$ is said to be Ricci flow invariant if for any
  $\R$ in the boundary $\partial\mathcal{C}$ of $\mathcal{C}$,
  $Q(\R)\in T_{\R}\mathcal{C}$, the tangent cone at $\R$ to $\mathcal{C}$.
\end{dfn}
\begin{rem}
  In other words, a cone is Ricci flow invariant if at every
  $\R\in\partial\mathcal{C}$, $Q(\R)$ points towards the inside of $\mathcal{C}$.

  This condition is equivalent to the fact that the solutions to the ODE
  $\frac{d}{dt}\R=Q(\R)$ which start inside $\mathcal{C}$ stay in
  $\mathcal{C}$ for \emph{positive} times.
\end{rem}

Hamilton's maximum principle (see \cite{zbMATH04022085}) implies:
\begin{thm}
  Let $\mathcal{C}$ be a Ricci flow invariant curvature cone.
  If $(M,g(t))_{t\in[0,T)}$ is a Ricci flow on a compact manifold such
  that $(M,g(0))$ has $\mathcal{C}$-nonnegative curvature, then for
  $t\in[0,T)$, $(M,g(t))$ has $\mathcal{C}$-nonnegative curvature.
\end{thm}
\begin{rem}
  It could happen that a nonnegativity condition is preserved under
  the Ricci flow while the associated cone is not Ricci flow invariant
  according to our definition, however such examples are not known to
  exist, as far as the knowledge of the author goes.
\end{rem}

\section{Examples of Ricci flow invariant curvature cones}
\label{sec:exemples-ricci-flow}

\subsection{First examples}
\label{sec:first-examples}

We will first give two prototypes of Ricci flow invariant cones. The
discovery of the Ricci flow invariance of these cones is due to Hamilton.

\begin{prop}
  $\mathcal{C}_{\scal}=\{\R|\tr \R\geq 0\}$ is a Ricci flow invariant nonnegative
  curvature cone.
\end{prop}
\begin{proof}
  The boundary of $\mathcal{C}_{\scal}$ is the hyperplane
  $\partial\mathcal{C}_{\scal}=\{\R|\tr \R=0\}$. The tangent cone at
  any $\R\in \partial\mathcal{C}_{\scal}$ 
  is actually (since $\mathcal{C}_{\scal}$ is a half space)
  $\mathcal{C}_{\scal}$ itself. 

  Thus we only need to show that for any
  $\R$ whose trace vanishes, $\tr Q(\R)\geq 0$. This is easily seen to
  be true thanks to the following formula (which is actually valid
  for any $\R$):
  \[\tr Q(\R)=2|\rho(\R)|^2\geq 0.\]
  (Recall that $\rho:S^2_B\Lambda^2\mathbb{R}^n\to
  S^2\mathbb{R}^n$ is the map which sends a curvature operator
  to its Ricci endomorphism.)
\end{proof}

\begin{prop}
  $\mathcal{C}_{PCO}=\{\R|\forall\omega\in\Lambda^2\mathbb{R}^2,\
  \ps{\R\omega}{\omega}\geq 0\}$ is a Ricci flow invariant nonnegative
  curvature cone.
\end{prop}
\begin{proof}
  A nonnegative curvature operator $\R$ belongs to the boundary of
  $\mathcal{C}_{PCO}$ if and only if $\ke(\R)\neq\{0\}$. In this
  case:
  \[T_{\R} \mathcal{C}_{PCO}=\{\Lop|\forall \omega\in \ke(\R),\
  \ps{\Lop\omega}{\omega}\geq 0\}.\]

  Thus we only need to show that for any $\omega\in\ke(\R)$,
  $\ps{Q(\R)\omega}{\omega}\geq 0$. Recall that $Q(\R)=\R^2+\R^\#$. By
  the symmetry of $\R$, $\ps{\R^2\omega}{\omega}=0$, thus we only need
  to deal with $\ps{\R^\#\omega}{\omega}$. In order to do this we
  choose an orthonormal basis $\eta_i$ of $\Lambda^2\mathbb{R}^n$
  consisting of eigenvectors $\R$ with associated eigenvalues
  $\lambda_i\geq 0$. We compute:
  \begin{align*}
    \ps{\R^\#\omega}{\omega}&=-\frac{1}{2}\tr (\ad_\omega\circ\R\circ
    \ad_\omega\circ\R)\\
    &=-\frac{1}{2}\sum_i\ps{\left[\omega,\R[\omega,\R\eta_i]\right]}{\eta_i}\\
    &=\frac{1}{2}\sum_i\ps{\R[\omega,\R\eta_i]}{[\omega,\eta_i]}\\
    &=\frac{1}{2}\sum_i\lambda_i\ps{\R[\omega,\eta_i]}{[\omega,\eta_i]}\\
    &\geq 0
  \end{align*}
  since $\R$ is nonnegative.
\end{proof}

Let us now briefly discuss the status of the most important curvature
cones in Riemannian geometry: the cone of operators with nonnegative
sectional curvature and the cone of curvature operators with
nonnegative Ricci curvature. They are Ricci flow invariant in dimension
3. However starting with dimension 4 these cone are not Ricci flow invariant. 

For the cone $\mathcal{C}_{\ric}$ of curvature operators with nonnegative Ricci curvature,
an even stronger result is 
actually available: there exists compact Kaehler surfaces $(M^4,g_0)$
with nonnegative Ricci curvature whose Ricci flow $(g(t))_{t>0}$
has negative Ricci curvature in some directions. This has been proven
by Maximo in \cite{zbMATH05863854}.

For the cone $\mathcal{C}_{\sec}$ of curvature operators with
nonnegative sectional curvature, it is quite easy to find an explicit
point in the boundary of $\mathcal{C}_{\sec}$ where $Q$ doesn't point
inside $\mathcal{C}_{\sec}$. 

Let $\R$ be the curvature operator of
$\mathbb{CP}^n$ ($n\geq 2$) normalized to have sectional curvature
between $1$ and $4$. Then $\tilde{\R}=\R-\Id\in\partial
\mathcal{C}_{\sec}$. Let $\Pi$ be a plane in $\mathbb{C}^n$ whose
sectional curvature is $1$. Let $(e_1,e_2)$ be a (real) basis of $\Pi$ (note
that $\mathbb{C}e_1$ and $\mathbb{C}e_2$ must be orthogonal in order
to ensure that the sectional curvature of $\Pi$ for $\R$ is $1$). The
sectional of $\Pi$ for $\tilde{\R}$ is $0$. Thus, if $Q(\tilde{\R})$
was to point inside $\mathcal{C}_{\sec}$, the sectional curvature of
$\Pi$ for $Q(\tilde{\R})$ would be nonnegative. But a quick computation
shows that $Q(\tilde{\R})=2(n+1)\R-(2n+5)\Id$. Hence the sectional
curvature of $\Pi$ with respect to $Q(\tilde{\R})$ is actually
$-3$. This shows that $\mathcal{C}_{\sec}$ is not Ricci flow invariant
starting in dimension $4$.

Thus starting in dimension 4, we need to consider more exotic
curvature cones. We will here define the most important of these Ricci
flow invariant curvature cones.

We start with the cone of $2$-nonnegative curvature operators:
\begin{dfn}
  $\mathcal{C}_{2PCO}$ is the cone of $2$-nonnegative curvature
  operators, more precisely it consists of all curvature operators
  whose two lowest eigenvalues have positive sum.
\end{dfn}

The other cones are derived from the ``positive isotropic curvature'' (PIC)
condition, introduced by Micallef and Moore in \cite{zbMATH04080313} as an obstruction to the
existence of area minimizing two spheres in compact Riemannian manifolds. 

We extend any curvature operator $\R$ to $\Lambda^2\mathbb{C}^n$ in a
complex linear way. We also extend the inner products on
$\mathbb{R}^n$ and $\Lambda^2\mathbb{R}^n$ in a complex bilinear
way to $\mathbb{C}^n$ and $\Lambda^2\mathbb{C}^n$. The resulting complex symmetric bilinear forms will still be
denoted by $\ps{\ }{\ }$. Note that these symmetric complex bilinear
forms admit isotropic vectors. A subspace $V\subset\mathbb{C}^n$ is
said to be totally isotropic if every very vector $v\in V$ is
isotropic (satisfies $\ps{v}{v}=0$).

The complex sectional curvature of a \emph{complex} plane $\Pi$
in $\mathbb{C}^n$ is defined by:
\[K^\mathbb{C}(\Pi)=\ps{\R(u\wedge v)}{\bar{u}\wedge\bar{v}}\]
where $(u,v)$ is a basis of $\Pi$ which is orthonormal basis with
respect to the hermitian 
inner product on $\mathbb{C}^n$ induced by $\ps{\ }{\ }$ and complex
conjugation. 

\begin{dfn}
  A curvature operator $\R$ is said to have nonnegative isotropic curvature
  (in short: is $\overline{PIC}$) if $K^\mathbb{C}(\Pi)\geq 0$ for
  every totally isotropic complex plane $\Pi$. The cone of
  $\overline{PIC}$ operators is denoted by $\mathcal{C}_{IC}$.
\end{dfn}
\begin{rem}
  The following characterization of PIC curvature operators is useful:
  $\R$ is $\overline{PIC}$ if and only if for every orthonormal
  $4$-frame $(e_1,e_2,e_3,e_4)$:
  \[\R_{1313}+\R_{1414}+\R_{2323}+\R_{2424}-2\R_{1234}\geq 0 \]
  where $R_{ijkl}$ stands for $\ps{\R(e_i\wedge e_j)}{e_k\wedge e_l}$.
\end{rem}

\begin{dfn}
  A curvature operator $\R:\Lambda^2\mathbb{R}^n\to \Lambda^2\mathbb{R}^n$ is
  $\overline{PIC1}$ if 
  its natural extension $\tilde{\R}:
  \Lambda^2\mathbb{R}^{n+1}\to \Lambda^2\mathbb{R}^{n+1}$ is
  $\overline{PIC}$. The cone of 
  $\overline{PIC1}$ operators is denoted by $\mathcal{C}_{IC1}$.
\end{dfn}
\begin{dfn}
  A curvature operator $\R:\Lambda^2\mathbb{R}^n\to \Lambda^2\mathbb{R}^n$ is
  $\overline{PIC2}$ if 
  its natural extension $\tilde{\R}:
  \Lambda^2\mathbb{R}^{n+2}\to \Lambda^2\mathbb{R}^{n+2}$ is
  $\overline{PIC}$. The cone of 
  $\overline{PIC2}$ operators is denoted by $\mathcal{C}_{IC2}$.
\end{dfn}
One obviously has that
$\mathcal{C}_{IC2}\subset\mathcal{C}_{IC1}\subset\mathcal{C}_{IC}$. 

\begin{thm}
  The nonnegative curvature cones $\mathcal{C}_{2PCO}$,
  $\mathcal{C}_{IC}$, $\mathcal{C}_{IC1}$ and $\mathcal{C}_{IC2}$ are
  Ricci flow invariant.
\end{thm}
The invariance of $\mathcal{C}_{2PCO}$ is due to Chen (\cite{zbMATH00030584}). The invariance
of $\mathcal{C}_{IC}$ is due to Brendle and Schoen
(\cite{zbMATH05859406}) and, independently, 
Nguyen (\cite{zbMATH05681042}). The invariance of $\mathcal{C}_{IC1}$ and
$\mathcal{C}_{IC2}$ follows immediately from the invariance of
$\mathcal{C}_{IC}$ and the fact that the ODE $\R'=Q(\R)$ respects
product structures. It was first pointed out in \cite{zbMATH05859406}.

We will see in the next section a unified proof of the invariance of
these four cones, due to Wilking.

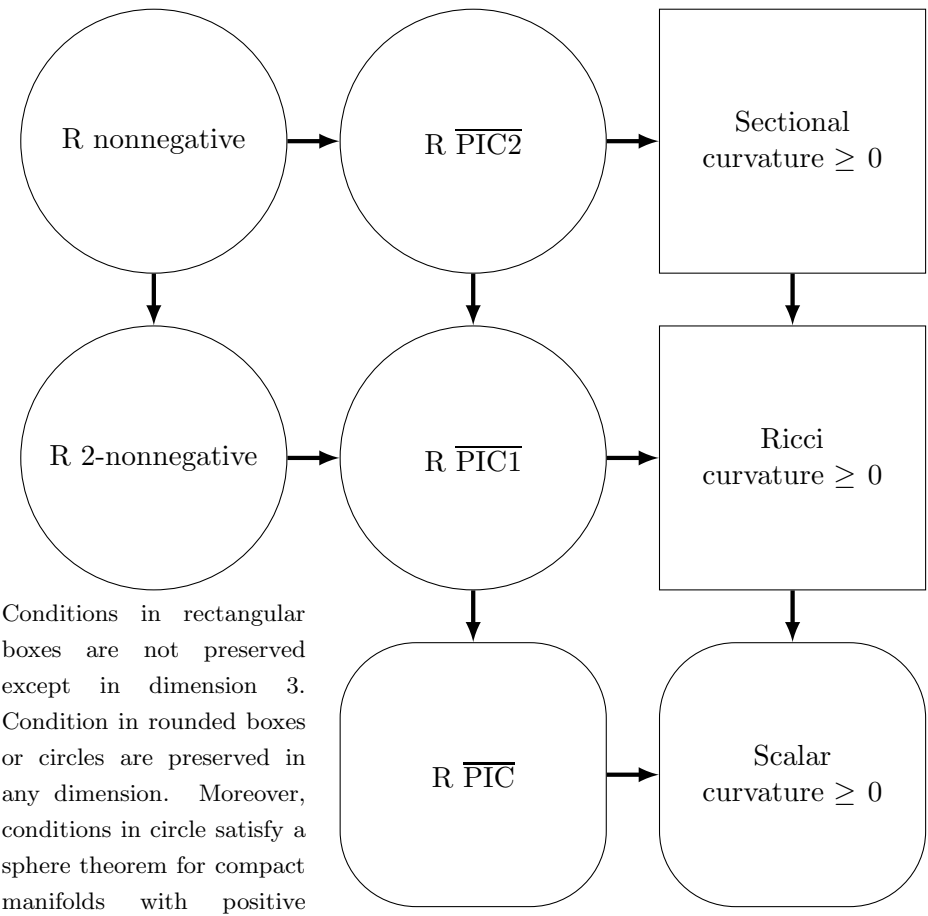
\begin{figure}
  \centering
    \tikzstyle{cond_non_pre}=[rectangle, draw, minimum size=3.5cm,text
    width=3cm, align=center]
    \tikzstyle{cond_pre}=[rectangle, rounded corners=1cm, draw,
    minimum size=3.5cm,text width=3cm, align=center]
    \tikzstyle{cond_pre_sph}=[circle, draw, minimum size=3.5cm, text
    width=3cm, align=center] 
    \tikzstyle{fleche}=[->,>=latex,ultra thick]
  \begin{tikzpicture}[node distance=4.2cm,scale=0.4]
    \node[cond_pre_sph] (PCO) {$\R$ nonnegative};
    \node[cond_pre_sph] (PIC2) [right of=PCO] {$\R$ \textoverline{PIC2}};
    \node[cond_pre_sph] (PCO2) [below of=PCO] {$\R$ 2-nonnegative};
    \node[cond_pre_sph] (PIC1) [below of=PIC2] {$\R$ \textoverline{PIC1}};
    \node[cond_non_pre] (SEC) [right of=PIC2] {Sectional curvature $\geq 0$};
    \node[cond_non_pre] (RIC) [right of=PIC1] {\nohyphens{Ricci curvature $\geq 0$}};
    \node[cond_pre] (PIC) [below of=PIC1] {$\R$ \textoverline{PIC}};
    \node[cond_pre] (SCAL) [right of=PIC] {\nohyphens{Scalar curvature $\geq
      0$}}; 
    \node[text width=4cm, align=justify] (LEG) [below of=PCO2]
    {{\footnotesize \nohyphens{Conditions in rectangular boxes are not preserved
        except in dimension 3.
      Condition in rounded boxes or circles are preserved in any
      dimension.  Moreover, conditions in circle satisfy a sphere
      theorem for compact manifolds with positive $\mathcal{C}$-curvature.}}};
    \draw[fleche] (PCO) to (PCO2);
    \draw[fleche] (PCO) to (PIC2);
    \draw[fleche] (PIC2) to (SEC);
    \draw[fleche] (PIC2) to (PIC1);
    \draw[fleche] (SEC) to (RIC);
    \draw[fleche] (PCO2) to (PIC1);
    \draw[fleche] (PIC1) to (RIC);
    \draw[fleche] (PIC1) to (PIC);
    \draw[fleche] (RIC) to (SCAL);
    \draw[fleche] (PIC) to (SCAL);
  \end{tikzpicture}
  \caption{Behavior of classical curvature conditions under Ricci flow. }
  \label{fig:implic_courbures}
\end{figure}

\subsection{Wilking's construction}
\label{sec:wilk-constr}
We now describe a construction due to Wilking which recovers most of the Ricci
flow invariant curvature conditions in one relatively easy proof. This
construction was published in \cite{zbMATH06185842}.

Before stating the criterion, recall that $\Lambda^2\mathbb{R}^n$ is
naturally isomorphic to $\mathfrak{so}(n,\mathbb{R})$. And not the
action of $SO(n,\mathbb{R})$ on $\Lambda^2\mathbb{R}^n$ is actually
just the adjoint action of $SO(n,\mathbb{R})$ on its Lie algebra
$\mathfrak{so}(n,\mathbb{R})$, thus we will denote the action an
element $g\in SO(n,\mathbb{R})$ on
$\omega\in\mathfrak{so}(n,\mathbb{R})\simeq\Lambda^2\mathbb{R}^n$ by 
$\Ad_g\omega$, moreover we will denote the Lie bracket
$[\omega,\eta]=\ad_\omega\eta$. 

We will also consider the action of
the complex Lie group $SO(n,\mathbb{C})$ on
$\mathfrak{so}(n,\mathbb{C})\simeq\Lambda^2\mathbb{C}^n$. 

\begin{dfn}
  Let $S$ be a subset of the complex Lie algebra $\mathfrak{so}
  (n,\mathbb{C})\simeq\Lambda^2\mathbb{C}^n$ which
  is invariant under the action of $SO(n,\mathbb{C})$. The nonnegative
  curvature cone 
  \[\mathcal{C}(S)= \left\{ \R \in S_B^2\Lambda^2\mathbb{R}^n\ \middle\vert
    \ \ps{\R(\omega)}{\bar{\omega}}\ge 0 \ \text{ for  all}\ \omega \in S \right\}\] 
  is called the \emph{Wilking Cone} associated with $S$.
\end{dfn}

\begin{thm}[\cite{zbMATH06185842}]
  Any Wilking cone $\mathcal{C}(S)$ is Ricci flow invariant.
\end{thm}
\begin{proof}[Sketch of proof:]
  Let $\R\in\partial\mathcal{C}(S)$, and let $\omega$ be any element
  of $S$ such that $\ps{\R\omega}{\bar{\omega}}$. We need to show that
  $\ps{Q(\R)\omega}{\bar{\omega}}\geq 0$. We obviously have that
  $\ps{\R^2\omega}{\bar{\omega}}=\ps{\R\omega}{\overline{\R\omega}}\geq
  0$. 

  We now show that $\ps{\R^\#\omega}{\bar{\omega}}\geq 0$. Let $\eta$
  be any element of $\mathfrak{so}(n,\mathbb{C})$, and consider the
  function:
  \[t\mapsto\ps{\R\Ad_{e^{t\eta}}\omega}{\overline{\Ad_{e^{t\eta}}\omega}}.\]
  It is nonnegative and attain its minimum at $t=0$. Differentiating
  twice and evaluating at $t=0$, we get that:
  \[\ps{\R(\ad_\eta\ad_\eta\omega}{\bar{\omega}}
  +2\ps{\R\ad_\eta\omega}{\ad_{\bar{\eta}}\bar{\omega}}
  +\ps{\R\omega}{\ad_{\bar{\eta}}\ad_{\bar{\eta}}\bar{\omega}}\geq 0. \]
  Replacing $\eta$ by $i\eta$ and summing we get that:
  \[\ps{\R\ad_\eta\omega}{\ad_{\bar{\eta}}\bar{\omega}}\geq 0.\]
  Hence:
  \[\ps{\R\ad_\omega\eta}{\ad_{\bar{\omega}}\bar{\eta}}\geq 0.\]
  This shows two things. First the Hermitian operator
  $-\ad_{\bar{\omega}}\circ\R\circ\ad_\omega$ and its conjugate
  $\Lop=-\ad_{\omega}\circ\R\circ\ad_{\bar{\omega}}$ are
  nonnegative. Second, $\R$ is nonnegative as an Hermitian operator on
  the image of $\ad_\omega$ (which contains the image of $\Lop$). 
  
  These two things together imply that:
  \[\ps{\R^\#\omega}{\bar{\omega}}=-\frac{1}{2}\tr
  (\Lop\R)=-\frac{1}{2}\tr (\R\Lop)\geq 0.\]
\end{proof}

Let us see now how this can be used to recover the Ricci flow
invariance of various curvature cones (see \cite{zbMATH06185842} for
more details):
\begin{itemize}
\item Choosing $S$ to be the whole $\mathfrak{so}(n,\mathbb{C})$, we
  recover the invariance of $\mathcal{C}_{PCO}$.
\item If we let $S=\{\omega\in
  \mathfrak{so}(n,\mathbb{C})\ |\ \omega^2=0\}$, we get the invariance of $\mathcal{C}_{2PCO}$.
\item With $S=\{\omega\in
  \mathfrak{so}(n,\mathbb{C})\ |\ \omega^2=0,\ \rg\omega=2\}$, we have the
  invariance of $\mathcal{C}_{IC}$. 
\item Letting $S=\{\omega\in
  \mathfrak{so}(n,\mathbb{C})\ |\ \omega^3=0,\ \rg\omega=2\}$, we directly get the
  invariance of $\mathcal{C}_{IC1}$.
\item Finally, with $S=\{\omega\in
  \mathfrak{so}(n,\mathbb{C})\ |\  \rg\omega=2\}$, we obtain the
  invariance of $\mathcal{C}_{IC2}$.
\end{itemize}

\subsection{One parameter families and differentiable sphere theorems}
\label{sec:one-param-famil}

Following the seminal work of Hamilton, one of the great successes of
Ricci flow has been the proof of ``differentiable sphere
theorems'' under various positive curvature assumptions. 
All these theorems are of the following form: \emph{``Let $(M,g_0)$ be a compact manifold
with positive $\mathcal{C}$-curvature, then $M$ admits a constant
sectional curvature metric and is thus diffeomorphic to a spherical
space form''}.

The proof of this kind of theorem using Ricci flow has a mandatory step,
the construction of a so called pinching set:
\begin{dfn}
  A closed convex $O(n)$-invariant $F\subset
  S^2_B\Lambda^2\mathbb{R}^n$ is a called  a pinching set if:
  \begin{itemize}
  \item It is Ricci flow invariant.
  \item As $\lambda$ goes to $0$, ${\lambda}F$ converges\footnote{By
      this we mean that for every compact $K\subset
      S^2_B\Lambda^2\mathbb{R}^n$, $(\lambda F)\cap K$ converges in
      the Hausdorff topology to $(\mathbb{R}_+)\Id\cap K$.}   
  to $\mathbb{R}_+\Id$.
  \end{itemize}
\end{dfn}
\begin{rem}
  A pinching set has to be a subset of the half space $\{\R|\tr\R> 0\}$.
\end{rem}

\begin{dfn}
  A curvature cone $\mathcal{C}$ is said to have the pinching property
  if for any compact $K$ contained in the \emph{interior} of
  $\mathcal{C}$, there is a pinching set $F$ containing $K$.
\end{dfn}
\definecolor{light-gray}{gray}{0.90}
\begin{figure}
  \centering
  \begin{tikzpicture}
    \node at (0,6) [above] {$\mathbb{R}_+\Id$};
    \clip (-4,6) rectangle (4,0);
    \draw (-4,0) -- (4,0);
    \draw[->] (0,0) -- (0,6);
    \draw[dashed] (-4,5) -- (0,0) -- (4,5); 
    \node at (-3,3) {$\mathcal{C}$};
    \draw[very thick,domain=-3:3] plot (\x,{0.5+\x*\x});
    \node at (1.3,3) {$F$};
    \draw[fill=light-gray] (0,2) circle (1);
    \node at (0.3,2) {$K$};    
  \end{tikzpicture}
  \caption{A pinching set $F$ containing a compact $K$ inside a cone $\mathcal{C}$ with the pinching property.}
  \label{fig:pinching_set}
\end{figure}
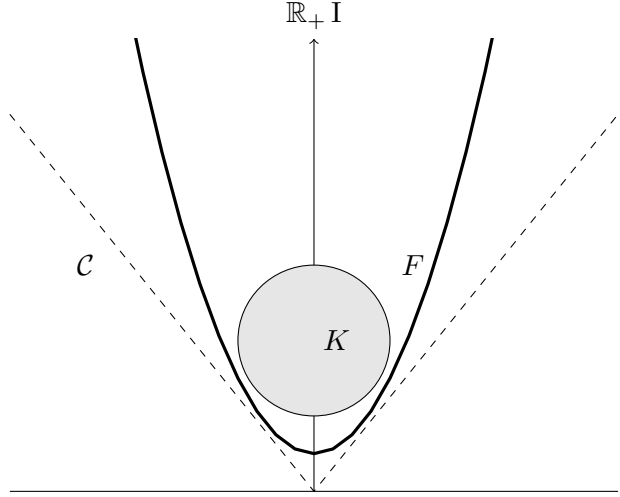
The interest of these definitions lies in the following proposition:
\begin{prop}[\cite{zbMATH04022085},\cite{zbMATH05578712}]
  Let $\mathcal{C}$ be curvature cone with the pinching property, then
  any compact manifold with positive $\mathcal{C}$-curvature
  admits a metric with constant positive sectionnal curvature.
\end{prop}
\begin{proof}[Sketch of proof:]
  Let $(M,g_0)$ be a compact manifold with positive
  $\mathcal{C}$-curvature. By compactness of $M$, there exist a
  compact $K$ contained in the interior of $\mathcal{C}$ such that
  $R_{g_0}(x)\in K$ for all $x\in M$. 

  Since $g_0$ has positive scalar curvature, the Ricci flow $g(t)$ starting
  at $g_0$ is defined only on some interval $[0,T)$ where $T<\infty$
  and the curvature of $g(t)$ blows up as $t\to T$. We consider a
  sequence of times $t_i<T$ converging to $T$, and points $x_i\in M$ where
  the maximum of the norm of the curvature tensor is attained. We let
  $C_i=\left |\R_{g(t_i)}\left(x_i\right)\right|$ and consider the
  sequence of parabolic blow ups centered at $x_i$:
  \[g_i(t)=C_ig_i\left(t_i+\frac{t}{C_i}\right).\]
  Perelman's no local collapsing theorem ensures that this sequence of
  Ricci flow converges, up to a subsequence, to some limiting Ricci
  flow $(M_\infty,g_\infty(t))$ defined on some time interval
  containing $0$.

  Let $F$ be a pinching set containing $K$, since $F$ is Ricci flow
  invariant, the curvature operator of
  of $g(t)$ satisfies $\R_{g(t)}(x)\in F$ for every $x\in M$ and $t\in
  [0,T)$, thus the curvature operator of $g_i(t)$ satisfies:
  \[\R_{g_i(t)}=\frac{1}{C_i}\R_{g{\left(t_i+\tfrac{t}{K_i}\right)}}\in\frac{1}{C_i}
  F.\]
  Letting $i$ go to infinity, we get that $\R_{g_\infty(t)}\in
  \mathbb{R}_+\Id$. Hence by Schur Lemma, $(M_\infty,g_\infty(t))$ has
  constant positive sectional curvature, end is thus compact. The
  compactness of $M_\infty$ implies that $M$ is diffeomorphic to $M_\infty$.
\end{proof}
\begin{rem}
  One should note that this proposition alone doesn't show that the
  Ricci flow converges (after a suitable normalisation) to a metric of
  constant curvature, this can actually be achieved but requires more
  work see for instance the book \cite{zbMATH05673930} by Brendle.
\end{rem}

The problem is thus reduced to the construction of suitable pinching
sets. A powerful method to achieve this is to build
``pinching families'' of curvature cones, which are $1$-parameter
families $(\mathcal{C}_s)_{s\in [0,1)}$ of curvature cones which start
at the curvature condition under consideration and pinch towards the
curvature cone $\mathbb{R}_+\Id$, which corresponds to constant
positive sectional curvature.

\begin{dfn}
  A one parameter family $(\mathcal{C}_{s})_{s\in [0,1)}$ of
  nonnegative curvature cones is called a
  pinching family if:
  \begin{itemize}
  \item $s\mapsto\mathcal{C}_s$ is continuous. 
  \item as $s\to 1$, $\mathcal{C}_s$ converges to $\mathbb{R}_+\Id$.
  \item for each $s>0$, every $\R\in\mathcal{C}_s\backslash\{0\}$ has
    positive trace.
  \item every $\mathcal{C}_s$ is Ricci flow invariant. Moreover for
    any $s>0$ and
    any non zero $\R\in\partial\mathcal{C}_s$, $Q(\R)$ belongs to the
    \emph{interior} of $T_{\R}\mathcal{C}_s$. \footnote{One could say
      that $Q$ points strictly towards the inside of $\mathcal{C}_s$.}
  \end{itemize}
\end{dfn}
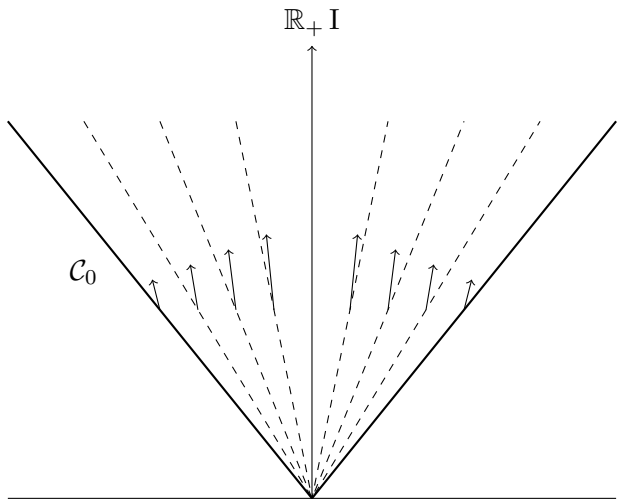
\begin{figure}
  \centering
  \begin{tikzpicture}
    \node at (0,6) [above] {$\mathbb{R}_+\Id$};
    \clip (-4,6) rectangle (4,0);
    \draw (-4,0) -- (4,0);
    \draw[->] (0,0) -- (0,6);
    \draw[thick] (-4,5) -- (0,0) -- (4,5); 
    \draw[->] (-2,2.5) -- +(-0.1,0.4);
    \draw[->] (2,2.5) -- +(0.1,0.4);
    \node at (-3,3) {$\mathcal{C}_0$};
    \draw[dashed] (-3,5) -- (0,0) -- (3,5); 
    \draw[dashed] (-2,5) -- (0,0) -- (2,5); 
    \draw[dashed] (-1,5) -- (0,0) -- (1,5); 
    \draw[->] (-1.5,2.5) -- +(-0.1,0.6);
    \draw[->] (-1,2.5) -- +(-0.1,0.8);
    \draw[->] (-0.5,2.5) -- +(-0.1,1);
    \draw[->] (1.5,2.5) -- +(0.1,0.6);
    \draw[->] (1,2.5) -- +(0.1,0.8);
    \draw[->] (0.5,2.5) -- +(0.1,1);
  \end{tikzpicture}
  
  \caption{A pinching family}
  \label{fig:pinching_family}
\end{figure}

The idea can be tracked back to Hamilton's
first paper on the Ricci flow of $3$-manifolds \cite{zbMATH03794902}, although it was not
really stressed. However, the
definition we give comes from the work of Boehm and Wilking \cite{zbMATH05578712}. In
this important paper, Boehm and Wilking also show the following
proposition, which says that the existence of a pinching family is
actually enough to ensure the existence of a pinching set:

\begin{thm}
  Let $(\mathcal{C}_s)_{s\in [0,1)}$ be a pinching family. Then
  $\mathcal{C}_0$ has the pinching property.
\end{thm}
\begin{rem}
  From a PDE point of view, the strength of this result can be seen in
  the following way: it allows to build from a family of scale
  invariant estimate (the pinching family) a self improving non scale
  invariant estimate (the pinching set). Usually, non scale invariant
  estimates are harder to obtain.
\end{rem}

Let us now give some examples of pinching families, the first example
is specific to dimension 3, and is implicit in Hamilton's original work:
\begin{prop}[\cite{zbMATH03794902}]
  The curvature cones \[\mathcal{C}_{s}=\left\{\R\in
    S^2_B\Lambda^2\mathbb{R}^3\ 
    \middle|\ \rho(\R)\geq\frac{2s}{3}(\tr\R)\id\geq 0\right\}\]
  for $s\in[0,1)$ form a pinching family such that $\mathcal{C}_0=\{\R|\rho(\R)\geq 0\}$.
\end{prop}
This gives a possible proof of Hamilton's result: any compact
$3$-manifold with positive Ricci curvature admits a metric with
constant sectional curvature.

Most of the differentiable sphere theorems obtained through Ricci
flow can be reduced to  the construction of suitable pinching families. (See for
instance \cite{zbMATH03979952}, \cite{zbMATH00702042},
\cite{zbMATH05578712}, \cite{zbMATH05859406}.)

The older previously cited works (\cite{zbMATH03979952},
\cite{zbMATH00702042}) build their pinching families through ad-hoc
constructions. A significant step forward was made by Boehm and
Wilking in \cite{zbMATH05578712}: they used $O(n)$-equivariant transformations formations  
\[\ell_{a,b}(\R)=\R_{\Id}+a\R_0+b\R_\mathcal{W}\]
to build pinching families from previously known Ricci flow invariant
curvature cones. 

More precisely, they considered, for a fixed Ricci
flow invariant curvature cone $\hat{\mathcal{C}}$, the cones:
\[\mathcal{C}_s=\left\{\R\middle|\ell_{a(s),b(s)}\left(\R\right)\in \hat{\mathcal{C}}\text{
  and }\rho(\R)\geq c(s)(\tr\R)\id\geq 0\right\} \]
and showed that, for suitable choices of $a(s),b(s)$ and $c(s)$, for
$\hat{\mathcal{C}}=\mathcal{C}_{2PCO}$, the cones $\mathcal{C}_s$ form
a pinching family.

This construction has been re-used by Brendle and Schoen in their
proof of the differentiable sphere theorem for $\tfrac{1}{4}$-pinched
manifold to build a pinching family starting at the cone
$\mathcal{C}_{IC2}$, see \cite{zbMATH05859406}. 

Let us end this section by mentioning Brendle's sphere theorem, in
the language we have developed it can phrased as:
the nonnegative curvature cone $\mathcal{C}_{IC1}$ has the pinching
property. However Brendle's proof is different from the strategy of
proof we have outlined here: it doesn't show the existence of a
pinching family starting at $\mathcal{C}_{IC1}$, but directly show the
existence of a pinching set. See \cite{zbMATH05859406} or
\cite{zbMATH05673930} for details.

\section{Restrictions on Ricci flow invariant curvature cones}
\label{sec:restr-ricci-flow}

The previous sections have stressed the importance of Ricci invariant
curvature cones, and showed which kind of examples are
available. Also, it should be noted that it is often computationally hard to
prove that a curvature cone is invariant. 

It is thus important to find
necessary conditions satisfied by every Ricci flow invariant curvature
cones, so that as many as possible curvature cones one can think of
can be discarded in advance without any computation.

This idea has not been investigated until recently, and only a handful of results are
available now. We present them here.

\subsection{Dimension independent restrictions}
\label{sec:dimens-indep-restr}

The first restrictions that were obtained dealt only with Wilking
cones. These can be found in \cite{zbMATH06176513} by Gururaja, Maity and Seshadri.

\begin{thm}
  Let $\mathcal{C}\subset S^2_B\Lambda^2\mathbb{R}^n$ be a Wilking
  with $n\geq 5$, then $\mathcal{C}$ is contained in the cone
  $\mathcal{C}_{IC}$ of manifolds with nonnegative isotropic curvature.
\end{thm}
The proof takes advantage the Wilking cone property by using some advanced
results on the orbits of the adjoint action of Lie groups. In the same
paper, the Wilking cones which had the pinching property were also
characterized:
\begin{thm}
  Let $\mathcal{C}\subset S^2_B\Lambda^2\mathbb{R}^n$ be a Wilking
  with $n\geq 5$, assume that $\mathcal{C}$ has the pinching property,
  then $\mathcal{C}$ is contained in the cone
  $\mathcal{C}_{IC1}$.  
\end{thm}

The next results do not require the Wilking assumptions. They come
from the article\cite{2013arXiv1308.1190R}, by H. Seshadri and the author.

\begin{thm}
  Let $\mathcal{C}\subset S^2_B\Lambda^2\mathbb{R}^n$, $n\geq 4$, be a
  Ricci flow invariant curvature cone, assume $\mathcal{C}$ contains
  $\mathcal{W}$, the space of Ricci flat curvature operators, then
  $\mathcal{C}$ is either the cone $\mathcal{C}_{\scal}$ or the whole
  space $S^2_B\Lambda^2\mathbb{R}^n$.
\end{thm}
\begin{rem}
  For representation theoretic reasons, any curvature cone which
  contains a non-vanishing tensor $W\in\mathcal{W}$ automatically
  contains the whole space $\mathcal{W}$. This will be useful in the applications.
\end{rem}
\begin{rem}
  Recall that nonnegative Ricci curvature is not preserved under Ricci
  flow in dimension 4 and above (see Section
  \ref{sec:first-examples}). It is thus quite natural to ask wether
  some weaker condition which is implied by nonnegative Ricci
  curvature is preserved, as the results proven using this condition
  would apply to the important class of manifolds with nonnegative
  Ricci curvature. The above result shows the only condition one can
  get following this idea is the ``nonnegative scalar curvature'' condition.
\end{rem}
\begin{proof}[Sketch of proof:]
  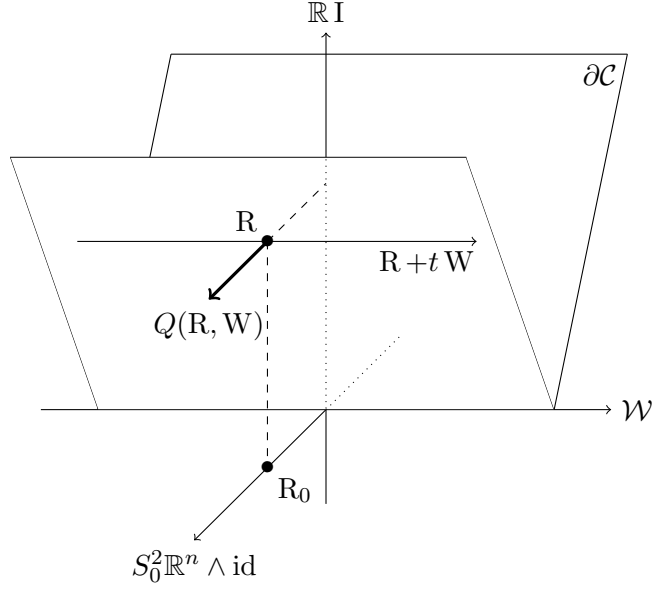
\begin{figure}[h]
  \centering
  \begin{tikzpicture}[scale=2.5]
    \coordinate (P1) at (-1.2,1.5,-1);
    \coordinate (PP1) at (-1.2,0,0);
    \coordinate (Q1) at (1.2,1.5,-1);
    \coordinate (QQ1) at (1.2,0,0);
    \coordinate (R1) at (-1.2,1.8,1.2);
    \coordinate (S1) at (1.2,1.8,1.2);
    \coordinate (L) at (0,1.2,0.8);
    \coordinate (LI) at (0,1.2,0);
    \coordinate (L0) at (0,0,0.8);    
    \coordinate (LpQLW) at (0,1.2,1.6);
    \coordinate (Lm) at (-1,1.2,0.8);
    \coordinate (Lp) at (1.1,1.2,0.8);
    \draw[->] (0,-0.5,0) -- (0,2,0) node[above] {$\mathbb{R}\Id$};
    \draw[->] (0,0,-1) -- (0,0,1.8) node[below] {$S^2_0\mathbb{R}^n\wedge\id$};
    \draw (P1) -- (PP1);
    \draw (QQ1) -- (Q1);
    \draw[->] (-1.5,0,0) -- (1.5,0,0) node[right] {$\mathcal{W}$};
    \draw (P1) -- (Q1);
    \draw (R1) -- (S1);
    \begin{scope}
      \clip  (R1) -- (S1) -- (QQ1) -- (PP1) -- cycle;
      \filldraw[fill=white] (R1) -- (S1) -- (QQ1) -- (PP1) -- cycle;
      \draw[dotted] (0,-0.5,0) -- (0,1.5,0) node[above] {$\mathbb{R}\Id$};      
      \draw[dotted] (0,0,-1) -- (0,0,1.8) node[below] {$S^2_0\mathbb{R}^n\wedge\id$};
    \end{scope}
    \draw (L) node {$\bullet$} node[above left] {$\R$};
    \draw (L0) node {$\bullet$} node[below right] {$\R_0$};
    \draw[dashed] (L) -- (LI);
    \draw[dashed] (L) -- (L0);
    \draw[->] (Lm) -- node[below,very near end] {$\R+t\Wop$} (Lp);
    \draw (Q1) node[below left] {$\partial\mathcal{C}$};
    \draw[very thick,->] (L) -- (LpQLW) node[below] {$Q(\R,\Wop)$};
 \end{tikzpicture}
  \caption{A cone which contains $\mathcal{W}$ must contain $S^2_0\mathbb{R}^n\wedge\id$.}
  \label{fig:curv_cone_W}
\end{figure}
  We argue by contradiction. Let $\mathcal{C}$ be a Ricci flow invariant curvature cone
  satisfying the hypothesis of the theorem but isn't one of the cones
  $\mathcal{C}_{\scal}$ or $S^2_B\Lambda^2\mathbb{R}^n$. Examining the
  decomposition of $S^2_B\Lambda^2\mathbb{R}^n$ as representation of
  $O(n)$ into irreducible components, we see that these hypothesis
  imply that $\mathcal{C}$ doesn't contain any non zero tensor in
  $S^2_0\mathbb{R}^n$ (see Figure \ref{fig:curv_cone_W}).
  
  The proof is based on the
  following observation: let $\R$ be a curvature operator in
  $\partial\mathcal{C}$, and $\Wop$ be a curvature operator in
  $\mathcal{W}$, then $\R+t\Wop\in\partial\mathcal{C}$ and the Ricci
  flow invariance gives that:
  \[Q(\R+t\Wop)\in T_{\R+t\Wop}\mathcal{C}=T_{\R}\mathcal{C}.\]
  We now compute:
  \[Q(\R+t\Wop)=Q(\R)+2tQ(\R,\Wop)+t^2Q(\Wop)\in T_{\R}\mathcal{C}.\]
  We remark that $Q(\Wop)\in\mathcal{W}\subset T_{\R}\mathcal{C}$.
  Since $\mathcal{W}$ is a vector space, $-Q(\Wop)\in
  T_{\R}\mathcal{C}$. We conclude that $Q(\R)+2tQ(\R,\Wop)$ belongs to
  $T_{\R}\mathcal{C}$ as a conical combination of two elements of
  $T_{\R}\mathcal{C}$. Dividing by $2t$ and letting $t$ go to
  infinity, we get that $Q(\R,\Wop)\in T_{\R}\mathcal{C}$.

  To end the proof we produce explicit operators $\R_0\in
  S^2_0\mathbb{R}^n\wedge\id$, $\Wop\in\mathcal{W}$ such that :
  \begin{itemize}
  \item There exists $a>0$ such that $\R=\Id+a\R_0$ belongs to $\mathcal{C}$ but
    $\R=\Id+(a+\eps)\R_0$ is outside of $\mathcal{C}$ for every $\eps>0$.
  \item $Q(\R,\Wop)=\R_0$.
  \end{itemize}
  This implies that $Q(\R,\Wop)$ which should be in
  $T_{\R}\mathcal{C}$, points \emph{outside} $\mathcal{C}$, this
  contradicts the Ricci flow invariance of $\mathcal{C}$. See Figure \ref{fig:curv_cone_W}.
\end{proof}

The following corollary can also be found in \cite{2013arXiv1308.1190R}:
\begin{cor}
  Let $\mathcal{C}\subset S^2_B\Lambda^2\mathbb{R}^n$, $n\geq 4$, be a
  Ricci flow invariant curvature cone, assume $\mathcal{C}$ contains
  in its \emph{interior} the curvature operator of a compact Einstein
  symmetric space with non constant nonnegative sectionnal curvature (such as
  $\mathbb{CP}^n$ or $\mathbb{S}^n\times\mathbb{S}^n$), then 
  $\mathcal{C}$ is either the cone $\mathcal{C}_{\scal}$ or the whole
  space $S^2_B\Lambda^2\mathbb{R}^n$.  
\end{cor}
\begin{rem}
  This theorem explains the following observation, whereas the complex
  projective space $\mathbb{CP}^n$ is the second example one usually
  gives when asked for a manifolds of positive curvature, when one
  considers Ricci flow invariant curvature conditions, it only has
  nonnegative curvature at best. The above theorem shows that
  ``nonnegative scalar curvature'' is the only Ricci flow invariant
  curvature condition for which $\mathbb{CP}^n$ is positively curved.
\end{rem}
\begin{rem}
  As in every \emph{even} dimension $2n$, with $n\geq 2$,
  $\mathbb{CP}^n$ is an Einstein symmetric space with non constant
  positive sectional curvature, the above theorem shows that in even
  dimensions there is no Ricci flow invariant curvature cone which
  contains the cone of operators with nonnegative sectional curvature.
\end{rem}
\begin{proof}[Sketch of proof:]
  Let $\R$ be the curvature operator of an Einstein symmetric space,
  a simple computation shows that $Q(\R)=\lambda\R$ (CITE BRENDLE). Since $\R$ is
  Einstein, we can decompose it as $\R=\R_{\Id}+\R_{\mathcal{W}}$. One
  then observes that solutions to $\R'=Q(\R)$ whose initial condition
  is of the form $\alpha\R_{\Id}+\beta\R_{\mathcal{W}}$ remain of this
  form. The components $\alpha$ and $\beta$ evolve according to the
  differential equations:
  
  \begin{equation}
  \begin{cases}
    \alpha'=\lambda\alpha^2\\
    \beta'=\lambda\beta^2.
  \end{cases}\label{eq:EDO}  
\end{equation}

  This system can be explicitly integrated and one sees from this that
  that all solutions escape to infinity and fall into three classes
  (see Figure \ref{fig:trajectories}):
  \begin{itemize}
  \item The diagonal is the trajectory of a solution.
  \item When the initial condition lies below the diagonal, the
    trajectory is part of an hyperbola which is asymptotic to a
    horizontal line.
  \item When the initial condition lies above the diagonal, the
    trajectory is part of an hyperbola which is asymptotic to a
    vertical line.
  \end{itemize}
  Now if we assume that $\R$ is in the interior of $\mathcal{C}$, then
  $\mathcal{C}$ must contain some trajectory whose initial condition
  is below the diagonal, since this trajectory will be asymptotic to
  an horizontal line, the fact that $\mathcal{C}$ is closed and convex implies that
  $\mathcal{C}$ must contain the ``$x$-axis''. Hence $\mathcal{C}$
  contains $\R_{\mathcal{W}}$ which is not zero because we assumed
  that $\R$ is the curvature operator of an Einstein symmetric space
  with non constant sectional curvature. Thus $\mathcal{C}$ contains
  $\mathcal{W}$ and we can apply the previous theorem to conclude that
  $\mathcal{C}$ is one of the cones $\mathcal{C}_{\scal}$ or the whole
  space $S^2_B\Lambda^2\mathbb{R}^n$.
  \begin{figure}[h]
    \centering
    \begin{tikzpicture}[scale=0.9]
      \draw[->] (0,0) -- (0,5) node[above] {$\R_{\Id}$};
      \draw[->] (0,0) -- (5,0) node[right] {$\R_{\mathcal{W}}$};
      \draw[->,domain=0:5] plot (\x,\x);
      \draw[->,domain=0:5] plot ({(2*\x)/(2+\x)},\x);
      \draw[->,domain=0:5] plot (\x,{(2*\x)/(2+\x)});
      \draw[->,domain=0:5] plot (\x,{(0.95*\x)/(0.96+0.05*\x)});
      \draw[->,domain=0:5] plot ({(0.95*\x)/(0.96+0.05*\x)},\x);
      \draw (2,2) node {$\bullet$};
      \draw (2,2) node[above left] {$\R$};
    \end{tikzpicture}
    \caption{Trajectories of $\R'=Q(\R)$ in the $(\R_{\Id},\R_{\mathcal{W}})$ plane,
      where $\R$ is the curvature operator of an Einstein symmetric
      space with non constant sectional curvature. (Equation \eqref{eq:EDO}.)}
    \label{fig:trajectories}
  \end{figure}
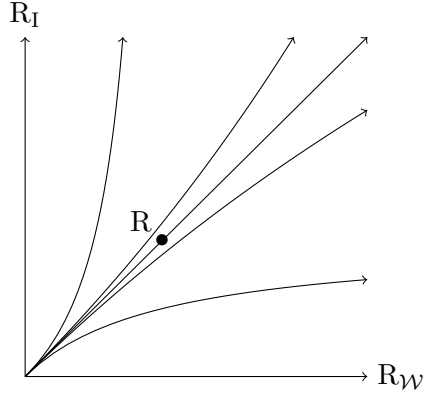
\end{proof}

\subsection{The case of dimension 4}
\label{sec:case-dimension-4}

Besides these dimension independent restrictions, we can prove much
stronger results in dimension 4.

When we consider $S^2_B\Lambda^2\mathbb{R}^4$ as a representation of
$O(4,\mathbb{R})$ we have the usual decomposition of
$S^2_B\Lambda^2\mathbb{R}^4$ into irreducible representations:
\[S^2_B\Lambda^2\mathbb{R}^4=\mathbb{R}\Id\oplus
(S^2_0\mathbb{R}^4\wedge\id)\oplus \mathcal{W}.
\]

However, when one considers only the action of $SO(4,\mathbb{R})$ on
$S^2_B\Lambda^2\mathbb{R}^4$, the component $\mathcal{W}$ is no longer 
irreducible and splits as
$\mathcal{W}=\mathcal{W}_+\oplus\mathcal{W}_-$, the self dual and anti
self dual Weyl tensors. This fact is specific
to dimension $4$ and can be linked to the fact that the Lie algebra
$\mathfrak{so}(n,\mathbb{R})$ is irreducible except when $n=4$ where
we have the splitting:
$\mathfrak{so}(4,\mathbb{R})\simeq\mathfrak{so}(3,\mathbb{R})\oplus\mathfrak{so}(3,\mathbb{R})$. 

This reducibility makes it interesting to weaken the requirements made
for curvature cones:
\begin{dfn}
  An \emph{oriented} curvature cone is a closed convex cone $\mathcal{C}\subset
  S^2_B\Lambda^2\mathbb{R}^4$ which is invariant under the action of
    $SO(4,\mathbb{R})$.
\end{dfn}
\begin{rem}
  We define a nonnegative oriented curvature cone and Ricci flow
  invariant oriented curvature cone in the same way as their
  unoriented counterparts. The main difference we will encounter is
  that $\mathcal{C}$-nonnegative curvature will now only make sense
  for \emph{oriented} Riemannian 4-manifolds.
\end{rem}
Let us give an example of a nonnegative oriented curvature cone which
is not a curvature cone:
\begin{dfn}
  A curvature operator $\R\in S^2_B\Lambda^2\mathbb{R}^4$ is to be
  \emph{positively half-$\overline{PIC}$} (in 
  short $\overline{PIC}_+$) if and only if for every \emph{positively oriented} orthonormal
  $4$-frame $(e_1,e_2,e_3,e_4)$:
  \begin{equation}
    \R_{1313}+\R_{1414}+\R_{2323}+\R_{2424}-2\R_{1234}\geq 0. \label{eqPICplus}  
  \end{equation}
  The oriented curvature cone of all such curvature operators will be
  denoted by $\mathcal{C}_{IC_+}$.

  Similarly, the cone $\mathcal{C}_{IC_-}$ of \emph{negatively half-$\overline{PIC}$}
  curvature operators is defined by requiring inequality
  \eqref{eqPICplus} to hold for every negatively oriented $4$-frame.
\end{dfn}
\begin{rem}
  The cones $\mathcal{C}_{IC_+}$ and $\mathcal{C}_{IC_-}$ lie between the cones
  $\mathcal{C}_{\scal}$ and $\mathcal{C}_{IC}$, and are Wilking
  cones. It can also be proven that they are Ricci flow invariant.
\end{rem}
\begin{ex}
  Besides $PIC$ $4$-manifolds such as $\mathbb{S}^4$ and
  $\mathbb{S}^3\times\mathbb{R}$, another example of $4$-manifold with
  strictly $\mathcal{C}_{IC_+}$-positive curvature is given by
  $\overline{\mathbb{CP}}^2$, the complex projective plane with the
  Fubini-Study metric and reversed orientation.
\end{ex}

In \cite{2013arXiv1311.5256R}, H. Seshadri and the author used the
results from \cite{2013arXiv1308.1190R} to prove that the cones
$\mathcal{C}_{IC_+}$ and $\mathcal{C}_{IC_-}$ enjoy some kind of
maximality among oriented Ricci flow invariant nonnegative curvature cones in
dimension 4.

\begin{thm}
  Let $\mathcal{C}\subset S^2_B\Lambda^2\mathbb{R}^4$ be an oriented
  curvature cone.
  \begin{itemize}
  \item If $\mathcal{C}\subsetneq \mathcal{C}_{\scal}$, then
    $\mathcal{C}$ is contained in either $\mathcal{C}_{IC_+}$ or $\mathcal{C}_{IC_-}$.
  \item If $\mathcal{C}$ is an (unoriented) curvature cone and
    $\mathcal{C}\subsetneq \mathcal{C}_{\scal}$, then 
    $\mathcal{C}$ is contained in $\mathcal{C}_{IC}$.
  \end{itemize}
\end{thm}

\section{Open questions}
\label{sec:open-questions}

We end this survey by some open questions about Ricci flow invariant
curvature cones. 

We have proved in section \ref{sec:dimens-indep-restr} that, in
dimension 4 and above, no Ricci
flow invariant cone can contain every curvature operator with nonnegative
Ricci curvature, except the cone $\mathcal{C}_{\scal}$. We also saw
that, in even dimension greater than 4, no Ricci flow invariant cone
can contain every curvature operator with nonnegative 
sectional curvature, except the cone $\mathcal{C}_{\scal}$.
It is thus natural to ask:
\begin{quest}
  In odd dimension 5 and above, does there exist Ricci flow invariant
  cones which contain every curvature operator with nonnegative
  sectional curvature ?
\end{quest}

In a similar but slightly more ambitious trend, one might want to
find the biggest non trivial Ricci flow invariant cone. We already
have a candidate for this:
\begin{quest}
What are the biggest Ricci flow invariant curvature cones ? Is there
any curvature cone, other that $\mathcal{C}_{\scal}$ which is not
contained the cone $\mathcal{C}_{IC}$ ?
\end{quest}
The last section answers this question affirmatively in dimension
4 (if one ignores oriented curvature cones). We also saw in the
beginning of section \ref{sec:dimens-indep-restr} that the answer is
yes when the investigation is restricted to Wilking cones.

Answering the following questions would make the matter of finding
sphere theorems using Ricci flow closed:
\begin{quest}
  Does there exist a Ricci flow invariant curvature cone with the
  pinching property which is not contained in $\mathcal{C}_{IC1}$ ? Is
  there a maximal Ricci flow invariant curvature cone with the
  pinching property ?
\end{quest}
We have seen in section \ref{sec:dimens-indep-restr} that
$\mathcal{C}_{IC1}$ is maximal among Wilking cones with the pinching
property. 

A Ricci flow with surgeries has been constructed for general
closed $3$-manifolds by Perelman and for closed $PIC$ 
$4$-manifolds by Hamilton. It is still unknown wether or not a Ricci
flow with surgeries can  be built for $PIC$ manifolds of arbitrary
dimensions. Next question is a possible step towards this goal:
\begin{quest}
  Can one pinch the cone $\mathcal{C}_{IC}$ towards a smaller cone ?
  For instance, can one find a continuous family of cones
  $(\mathcal{C}_s)_{s\in[0,1]}$ such that:
  \begin{itemize}
  \item $\mathcal{C}_0=\mathcal{C}_{IC}$,
  \item for each $s\in (0,1)$, $Q$ points strictly inwards $\mathcal{C}_s$,
  \item $\mathcal{C}_1$ is a ``geometrically constraining'' cone, such
    as $\mathcal{C}_{PCO}$ or $\mathcal{C}_{IC1}$.
  \end{itemize}
\end{quest}
The existence of such a ``generalized pinching family'' should imply
that blow up of singularities of $PIC$ Ricci flows should have
$\mathcal{C}_{1}$-nonnegative curvature, which would make their
study easier. This would play a role
similar to the role of the Hamilton-Ivey estimate in dimension 3.

It would also be interesting to see how these results can be adapted
to the Kaehler-Ricci flow. Note first that Wilking's construction
-works in the Kaehler case with only minor modifications (this is
treated in \cite{zbMATH06185842}). However, for results in the spirit
of section \ref{sec:restr-ricci-flow}, it seems that a better
understanding of how $Q$ interacts with the representation theory of
Kaehler curvature operators.

\bibliographystyle{alpha}
\bibliography{TSG-RF_Cones}

\newcommand{\etalchar}[1]{$^{#1}$}
\begin{thebibliography}{BBM{\etalchar{+}}10}

\bibitem[AH11]{zbMATH05821328}
Ben {Andrews} and Christopher {Hopper}.
\newblock {\em {The Ricci flow in Riemannian geometry. A complete proof of the
  differentiable 1/4-pinching sphere theorem.}}
\newblock Berlin: Springer, 2011.

\bibitem[BBM{\etalchar{+}}10]{zbMATH05796326}
Laurent {Bessi\`eres}, G\'erard {Besson}, Sylvain {Maillot}, Michel {Boileau},
  and Joan {Porti}.
\newblock {\em {Geometrisation of 3-manifolds.}}
\newblock Z\"urich: European Mathematical Society (EMS), 2010.

\bibitem[{Bre}10]{zbMATH05673930}
Simon {Brendle}.
\newblock {\em {Ricci flow and the sphere theorem.}}
\newblock Providence, RI: American Mathematical Society (AMS), 2010.

\bibitem[BS09]{zbMATH05859406}
Simon {Brendle} and Richard {Schoen}.
\newblock {Manifolds with $1/4$-pinched curvature are space forms.}
\newblock {\em {J. Am. Math. Soc.}}, 22(1):287--307, 2009.

\bibitem[BW08]{zbMATH05578712}
Christoph {B\"ohm} and Burkhard {Wilking}.
\newblock {Manifolds with positive curvature operators are space forms.}
\newblock {\em {Ann. Math. (2)}}, 167(3):1079--1097, 2008.

\bibitem[{Che}91]{zbMATH00030584}
Haiwen {Chen}.
\newblock {Pointwise 1/4-pinched 4-manifolds.}
\newblock {\em {Ann. Global Anal. Geom.}}, 9(2):161--176, 1991.

\bibitem[CK04]{zbMATH02121403}
Bennett {Chow} and Dan {Knopf}.
\newblock {\em {The Ricci flow: an introduction.}}
\newblock Providence, RI: American Mathematical Society (AMS), 2004.

\bibitem[CLN06]{zbMATH05081815}
Bennett {Chow}, Peng {Lu}, and Lei {Ni}.
\newblock {\em {Hamilton's Ricci flow.}}
\newblock Providence, RI: American Mathematical Society (AMS), 2006.

\bibitem[GMS13]{zbMATH06176513}
H.A. {Gururaja}, Soma {Maity}, and Harish {Seshadri}.
\newblock {On Wilking's criterion for the Ricci flow.}
\newblock {\em {Math. Z.}}, 274(1-2):471--481, 2013.

\bibitem[{Gro}91]{zbMATH00709417}
M.~{Gromov}.
\newblock {Sign and geometric meaning of curvature.}
\newblock {\em {Rend. Semin. Mat. Fis. Milano}}, 61:9--123, 1991.

\bibitem[{Ham}82]{zbMATH03794902}
Richard~S. {Hamilton}.
\newblock {Three-manifolds with positive Ricci curvature.}
\newblock {\em {J. Differ. Geom.}}, 17:255--306, 1982.

\bibitem[{Ham}86]{zbMATH04022085}
Richard~S. {Hamilton}.
\newblock {Four-manifolds with positive curvature operator.}
\newblock {\em {J. Differ. Geom.}}, 24:153--179, 1986.

\bibitem[{Hui}85]{zbMATH03979952}
Gerhard {Huisken}.
\newblock {Ricci deformation of the metric on a Riemannian manifold.}
\newblock {\em {J. Differ. Geom.}}, 21:47--62, 1985.

\bibitem[KL08]{zbMATH05530173}
Bruce {Kleiner} and John {Lott}.
\newblock {Notes on Perelman's papers.}
\newblock {\em {Geom. Topol.}}, 12(5):2587--2855, 2008.

\bibitem[{Mar}94]{zbMATH00702042}
Christophe {Margerin}.
\newblock {Une caract\'erisation optimale de la structure diff\'erentielle
  standard de la sph\`ere en terme de courbure pour (presque) toutes les
  dimensions. I: Les \'enonc\'es.}
\newblock {\em {C. R. Acad. Sci., Paris, S\'er. I}}, 319(6):605--607, 1994.

\bibitem[{Max}11]{zbMATH05863854}
Davi {Maximo}.
\newblock {Non-negative Ricci curvature on closed manifolds under Ricci flow.}
\newblock {\em {Proc. Am. Math. Soc.}}, 139(2):675--685, 2011.

\bibitem[MM88]{zbMATH04080313}
Mario~J. {Micallef} and John~Douglas {Moore}.
\newblock {Minimal two-spheres and the topology of manifolds with positive
  curvature on totally isotropic two-planes.}
\newblock {\em {Ann. Math. (2)}}, 127(1):199--227, 1988.

\bibitem[MT07]{zbMATH05188193}
John {Morgan} and Gang {Tian}.
\newblock {\em {Ricci flow and the Poincar\'e conjecture.}}
\newblock Providence, RI: American Mathematical Society (AMS); Cambridge, MA:
  Clay Mathematics Institute, 2007.

\bibitem[{Ngu}10]{zbMATH05681042}
Huy~T. {Nguyen}.
\newblock {Isotropic curvature and the Ricci flow.}
\newblock {\em {Int. Math. Res. Not.}}, 2010(3):536--558, 2010.

\bibitem[{Per}02]{2002math.....11159P}
G.~{Perelman}.
\newblock {The entropy formula for the Ricci flow and its geometric
  applications}.
\newblock {\em ArXiv Mathematics e-prints}, November 2002.

\bibitem[{Per}03a]{2003math......7245P}
G.~{Perelman}.
\newblock {Finite extinction time for the solutions to the Ricci flow on
  certain three-manifolds}.
\newblock {\em ArXiv Mathematics e-prints}, July 2003.

\bibitem[{Per}03b]{2003math......3109P}
G.~{Perelman}.
\newblock {Ricci flow with surgery on three-manifolds}.
\newblock {\em ArXiv Mathematics e-prints}, March 2003.

\bibitem[RS13a]{2013arXiv1308.1190R}
T.~{Richard} and H.~{Seshadri}.
\newblock {Noncoercive Ricci flow invariant curvature cones}.
\newblock {\em ArXiv e-prints, to appear in Proc. of American Mathematical
  Society}, August 2013.

\bibitem[RS13b]{2013arXiv1311.5256R}
T.~{Richard} and H.~{Seshadri}.
\newblock {Positive isotropic curvature and self-duality in dimension 4}.
\newblock {\em ArXiv e-prints}, November 2013.

\bibitem[{Top}06]{zbMATH05062652}
Peter {Topping}.
\newblock {\em {Lectures on the Ricci flow.}}
\newblock Cambridge: Cambridge University Press, 2006.

\bibitem[{Wil}13]{zbMATH06185842}
Burkhard {Wilking}.
\newblock {A Lie algebraic approach to Ricci flow invariant curvature
  conditions and Harnack inequalities.}
\newblock {\em {J. Reine Angew. Math.}}, 679:223--247, 2013.

\end{thebibliography}
\end{document}